\documentclass[12pt,twoside,a4paper]{article}
\usepackage{amsmath,amsthm, amssymb, amsfonts,latexsym,amscd}
\usepackage[dvips]{graphicx}
\newtheorem{theorem}{Theorem}[section]
\newtheorem{corollary}[theorem]{Corollary}
\newtheorem{lemma}[theorem]{Lemma}
\newtheorem{proposition}[theorem]{Proposition}

\begin{document}

\title{\bf On the minimum cut-sets of the power graph of a finite cyclic group}
\author{Sanjay Mukherjee \and Kamal Lochan Patra \and Binod Kumar Sahoo}
%\date{}
\maketitle

\begin{abstract}
The power graph $\mathcal{P}(G)$ of a finite group $G$ is the simple graph with vertex set $G$, in which two distinct vertices are adjacent if one of them is a power of the other. For an integer $n\geq 2$, let $C_n$ denote the cyclic group of order $n$ and let $r$ be the number of distinct prime divisors of $n$. The minimum cut-sets of $\mathcal{P}(C_n)$ are characterized in \cite{cps} for $r\leq 3$. In this paper, for $r\geq 4$, we identify certain cut-sets of $\mathcal{P}(C_n)$ such that any minimum cut-set of $\mathcal{P}(C_n)$ must be one of them.\\

\noindent {\bf Key words:} Power graph, Cut-set, Cyclic group, Euler's totient function \\
{\bf AMS subject classification.} 05C25, 05C40
\end{abstract}

\section{Introduction}

Let $\Gamma$ be a simple graph with vertex set $V.$ A subset $X$ of $V$ is called a {\it cut-set} of $\Gamma$ if the induced subgraph of $\Gamma$ with vertex set $V\setminus X$ is  disconnected. A cut-set $X$ of $\Gamma$ is called a {\it minimum cut-set} if $|X|\leq |Y|$ for any cut-set $Y$ of $\Gamma$. If $V=A\cup B$ for two nonempty disjoint subsets $A, B$ of $V$ such that there is no edge of $\Gamma$ containing vertices from both $A$ and $B$, then we say that $A\cup B$ is a {\it separation} of $\Gamma$.
Note that $\Gamma$ is disconnected if and only if there exists a separation of $\Gamma$.

\subsection{Power graphs}

There are various graphs associated with groups which have been studied in the literature, e.g., Cayley graphs, commuting graphs. The notion of directed power graph of a group was introduced by Kelarev and Quinn in \cite{ker1}. The underlying undirected graph, simply termed as the `power graph' of that group, was first considered by Chakrabarty et al. in \cite{ivy}. Several researchers have then studied both the directed and undirected power graphs of groups from different viewpoints. We refer to the survey papers \cite{AKC, KSCC} and the references therein for more on these graphs.

The {\it power graph} of a finite group $G$, denoted by $\mathcal{P}(G)$, is the simple graph with vertex set $G$, in which two distinct vertices are adjacent if one of them is a power of the other. Since the identity element of $G$ is adjacent to all other vertices, $\mathcal{P}(G)$ is a connected graph of diameter at most two. Chakrabarty et al. proved that $\mathcal{P}(G)$ is a complete graph if and only if $G$ is a cyclic group of prime power order \cite[Theorem 2.12]{ivy} .

Cameron and Ghosh \cite{cam-1} and Mirzargar et. al \cite{mir, mir-1} have studied the isomorphism problem associated with finite groups and their power graphs. Feng et al. \cite{FMW-1} described the full automorphism group of the power graph of a finite group. Kirkland et al. \cite{kirkland} derived explicit formulas concerning the complexity (that is, number of spanning trees) of power graphs of various finite groups. Graph parameters such as chromatic number \cite{MaFeng2015,shitov}, metric dimension \cite{FMW}, strong metric dimension \cite{MaFengWang2018, Ma} etc. of power graphs of finite groups have also been studied in the literature. It was proved in \cite[Theorem 5]{dooser} and \cite[Corollary 2.5]{FMW} that the power graph of a finite group is perfect, in particular, the clique number and the chromatic number coincide. Explicit formula for the clique number of the power graph of a finite cyclic group was obtained in \cite[Theorem 2]{mir} and \cite[Theorem 7]{dooser}.

For a given finite group $G$, characterizing the minimum cut-sets of $\mathcal{P}(G)$ is an interesting problem. It was proved in \cite[Theorem 1.3]{cur} and \cite[Corollary 3.4]{cur-1} respectively that, among all finite groups of order $n$, the power graph $\mathcal{P}(C_n)$ of the cyclic group $C_n$ of order $n$ has the maximum number of edges and has the largest clique. It is thus expected that the size of a minimum cut-set of $\mathcal{P}(C_n)$ would be larger comparing to the size of the minimum cut-set of the power graph of a noncyclic group of order $n$. A few cut-sets of $\mathcal{P}(C_n)$ were identified in \cite{cps, cps-2, panda} which appeared to be minimum cut-sets for some values for $n$. In this paper, we identify certain cut-sets of $\mathcal{P}(C_n)$ such that one of them must be a minimum cut-set for any given $n$.

\subsection{Cut-sets of $\mathcal{P}(C_n)$}

Throughout the paper, $n\geq 2$ is an integer and $r$ denotes the number of distinct prime divisors of $n$. We write the prime power factorization of $n$ as
$$n=p_1^{n_1}p_2^{n_2}\cdots p_r^{n_r},$$
where $n_1,n_2,\ldots, n_r$ are positive integers and $p_1,p_2,\ldots,p_r$ are primes with $p_1<p_2<\cdots <p_r$. For a positive integer $m$, we denote by $[m]$ the set $\{1,2,\dots,m\}$.

For every positive divisor $d$ of $n$, there is a unique (cyclic) subgroup of $C_n$ of order $d$. It follows that two distinct elements $x,y$ of $C_n$ are adjacent in $\mathcal{P}(C_n)$ if and only if $o(x)\mid o(y)$ or $o(y)\mid o(x)$, where $o(z)$ denotes the order of an element $z\in C_n$. The lattice of all subgroups of $C_n$ with respect to inclusion is isomorphic to the lattice of all divisors of $n$ with respect to divisibility. If $H_1,H_2,\ldots ,H_k$ are subgroups of $C_n$, then the number of elements in the intersection $H_1\cap H_2\cap\ldots\cap H_k$ is equal to the greatest common divisor of the integers $|H_1|,|H_2|,\ldots ,|H_k|$.

For every positive divisor $d$ of $n$, we denote by $E_d$ the set of all elements of $C_n$ whose order is $d$ and by $S_d$ the set of all elements of $C_n$ whose order divides $d$.
Then $S_d$ is the unique (cyclic) subgroup of $C_n$ of order $d$ and $E_d$ is precisely the set of generators of $S_d$. Thus $\left\vert S_d\right\vert =d$ and $\left\vert E_d\right\vert =\phi(d)$, where $\phi$ is the Euler's totient function. Any two distinct elements of $E_d$ are adjacent in $\mathcal{P}(C_n)$. Any cut-set of $\mathcal{P}(C_n)$ must contain the sets $E_n$ and $E_1$, as each element from these two sets is adjacent with all other vertices of $\mathcal{P}(C_n)$.

\subsubsection*{\underline{Cut-sets of the first type}}

Let $r\geq 2$, $a\in [r]$ and $s\in [n_a]$. We denote by $Q_a^s$ the union of the subgroups $S_{\frac{n}{p_ip_a^s}}$ of $C_n$, where $i\in [r]\setminus\{a\}$. Define the subset $Z_a^s$ of $C_n$ by
$$Z_a^s:=E_n\cup E_{\frac{n}{p_a}}\cup E_{\frac{n}{p_a^2}}\cup\ldots\cup E_{\frac{n}{p_a^{s-1}}}\cup Q_a^s.$$
In Section \ref{sec-Z-a-s}, we prove that $Z_a^s$ is a cut-set of $\mathcal{P}(C_n)$.

\subsubsection*{\underline{Cut-sets of the second type}}

Let $r\geq 3$, $a,b\in [r]$ with $a\neq b$, $s\in [n_a]$ and $t\in [n_b]$. We denote by $K_{a,b}^{s,t}$ the set of all the nongenerators of the cyclic subgroup $S_{\frac{n}{p_a^s p_b^t}}$ of $C_n$ and by $H_{a,b}^{s,t}$ the union of the mutually disjoint sets $E_{\frac{n}{p_a^ip_b^j}}$, where $0\leq i\leq s$, $0\leq j\leq t$ and $(i,j)\neq (s,t)$. Define the subset $X_{a,b}^{s,t}$ of $C_n$ by
$$X_{a,b}^{s,t}:=H_{a,b}^{s,t}\bigcup K_{a,b}^{s,t}.$$
In Section \ref{sec-X-a-b-s-t}, we prove that $X_{a,b}^{s,t}$ is a cut-set of $\mathcal{P}(C_n)$.

\subsection{Main result}

If $r=1$, then $n$ is a prime power and so $\mathcal{P}(C_n)$ is a complete graph. In this case, there is no cut-set of $\mathcal{P}(C_n)$.

Suppose that $r=2$. Then $n=p_1^{n_1}p_2^{n_2}$. If $p_1 \geq 3$, then $2\phi(p_1) >p_1$ and so $Z_2^1$ is the only minimum cut-set of $\mathcal{P}(C_n)$ by \cite[Proposition 4.4]{cps}.
If $p_1=2$, then it follows from the proof of \cite[Theorem 1.3(iii)]{cps} that the number of distinct minimum cut-sets of $\mathcal{P}(C_n)$ is $n_2$, namely the cut-sets $Z_2^s$ for $s\in [n_2]$.

Suppose that $r=3$. Then $n=p_1^{n_1}p_2^{n_2}p_3^{n_3}$. If $p_1\geq 3$, then $2\phi(p_1p_2) >p_1p_2$ and so $Z_3^1$ is the only minimum cut-set of $\mathcal{P}(C_n)$ by \cite[Proposition 4.4]{cps}. If $p_1=2$, then it follows from the proof of \cite[Theorem 1.5]{cps} that $Z_3^{n_3}$ is the only minimum cut-set of $\mathcal{P}(C_n)$.

For $r\geq 4$, we study the possible minimum cut-sets of $\mathcal{P}(C_n)$ and prove the following theorem in Section \ref{sec-main-proof}.

\begin{theorem}\label{main}
Let $r\geq 4$ and $X$ be a minimum cut-set of $\mathcal{P}(C_n)$. Then $X=Z_r^1$; or $X=Z_a^{n_a}$ for some $a\in [r]$ with $n_a\geq 2$; or $X=X_{a,b}^{s,t}$ for some $a,b\in [r]$ with $a\neq b$, $s\in [n_a]$ and $t\in [n_b]$.
\end{theorem}

\section{Euler's totient function}

We continue with the notation introduced in the previous section. We have $\phi\left(p^k\right)=p^{k-1}(p-1)=p^{k-1}\phi(p)$ for any prime $p$ and positive integer $k$. Since $\phi$ is a multiplicative function, we have  $\phi(n)= \phi\left(p_1^{n_1}\right)\phi\left(p_2^{n_2}\right) \cdots \left(p_r^{n_r}\right)$. For $i,k\in [r]$ with $i\leq k$, we have
\begin{equation}\label{eqn-1}
\left|E_{\frac{n}{p_i}} \right|=\phi\left(\frac{n}{p_i}\right)\geq \phi\left(\frac{n}{p_k}\right)=\left|E_{\frac{n}{p_k}} \right|.
\end{equation}
Let $I$ be a nonempty subset of $[r]$ with $|I|=t$. Then by \cite[Lemma 3.1]{cps-1}, we have
\begin{equation}\label{eqn-2}
(t+1)\phi\left(\underset{i\in I}\prod p_i\right) \geq \underset{i\in I}\prod p_i,
\end{equation}
where equality holds if and only if $(t,I)=(1,\{1\})$ with $p_1=2$ or $(t,I)=(2,\{1,2\})$ with $p_1=2$, $p_2=3$.
By expanding $\phi\left(\underset{i\in I}\prod p_i \right)=\underset{i\in I}\prod  (p_i -1)$, it follows that
\begin{equation}\label{eqn-3}
\underset{i\in I}\prod p_i -\phi\left(\underset{i\in I}\prod p_i \right)= \underset{i=I}{\sum}\frac{\eta}{p_{i}} - \underset{i<j}{\underset{i,j\in I}{\sum}}\frac{\eta }{p_{i} p_{j}} + \underset{i<j<k}{\underset{i,j,k\in I}{\sum}}\frac{\eta }{p_ip_{j} p_{k}} -\cdots +(-1)^{t-1},
\end{equation}
where $\eta= \underset{i\in I}\prod p_i$. We next calculate the cardinality of certain subsets of $C_n$ which are needed in the subsequent sections.

\begin{lemma}\label{lem-2.1}
Let $a\in [r]$ and $0\leq k\leq s\leq n_a$. Then
$$
\underset{l=k}{\overset{s}{\sum}} \left\vert E_{\frac{n}{p_a^{l}}} \right\vert =\underset{l=k}{\overset{s}{\sum}} \phi\left( \frac{n}{p_a^{l}}\right)=
\begin{cases}
\frac{n}{p_1\cdots p_r}\times \frac{1}{p_a^{k-1}}\times \phi\left(\frac{p_1\cdots p_r}{p_a} \right), &\text{if $s=n_a$}\\
\frac{n}{p_1\cdots p_r}\times \frac{1}{p_a^{k-1}}\times\phi\left(\frac{p_1\cdots p_r}{p_a} \right)\left(1-\frac{1}{p_a^{s-k+1}}\right), &\text{if $s<n_a$}
\end{cases}.
$$
\end{lemma}

\begin{proof}
We have $\underset{l=k}{\overset{s}{\sum}} \phi\left( \frac{n}{p_a^{l}}\right) = \phi\left( \frac{n}{p_a^{n_a}}\right)\left[ \underset{l=k}{\overset{s}{\sum}} \phi\left(p_a^{n_a-l}\right)\right]$. Since
$\phi\left( \frac{n}{p_a^{n_a}}\right)= \frac{n}{p_1p_2\cdots p_r}\times \frac{1}{p_a^{n_a-1}}\times \phi\left(\frac{p_1p_2\cdots p_r}{p_a} \right)$ and
$$\underset{l=k}{\overset{s}{\sum}} \phi\left(p_a^{n_a-l}\right)=\underset{i=n_a-s}{\overset{n_a-k}{\sum}} \phi(p_a^i)= \left\{
\begin{array}{ll}
    p_a^{n_a-k}, & \text{if } s=n_a\\
    p_a^{n_a-k}\left(1-\frac{1}{p_a^{s-k+1}}\right), & \text{if } s<n_a,\\
\end{array}
\right.$$
the lemma follows.
\end{proof}

\begin{lemma}\label{lem-2.4}
Let $a,b\in [r]$ with $a\neq b$, $0\leq s\leq n_a$ and $0\leq t\leq n_b$. Then
$$\underset{k=s}{\overset{n_a}{\sum}}\; \underset{l=t}{\overset{n_b}{\sum}} \left\vert E_{\frac{n}{p_a^{k}p_b^{l}}} \right\vert
=\phi\left( \frac{n}{p_a^{n_a}p_b^{n_b}}\right) p_a^{n_a-s} p_b^{n_b-t}.$$
\end{lemma}

\begin{proof}
We have $\underset{k=s}{\overset{n_a}{\sum}} \;\underset{l=t}{\overset{n_b}{\sum}} \left\vert E_{\frac{n}{p_a^{k}p_b^{l}}} \right\vert = \underset{k=s}{\overset{n_a}{\sum}}\; \underset{l=t}{\overset{n_b}{\sum}} \phi\left( \frac{n}{p_a^{k}p_b^{l}}\right) =\phi\left(\frac{n}{p_a^{n_a}p_b^{n_b}}\right) \left[\underset{k=s}{\overset{n_a}{\sum}}\phi\left( p_a^{n_a -k}\right)\right] \left[\underset{l=t}{\overset{n_b}{\sum}} \phi\left( p_b^{n_b - l}\right) \right]$. Then the lemma follows as $\underset{k=s}{\overset{n_a}{\sum}}\phi\left( p_a^{n_a -k}\right) = p_a^{n_a-s}$ and  $\underset{l=t}{\overset{n_b}{\sum}} \phi\left( p_b^{n_b - l}\right)=p_b^{n_b-t}$.
\end{proof}

\begin{lemma}\label{lem-Q-a-s}
For $a\in [r]$ and $s\in [n_a]$, let $Q_a^s$ be the union of the subgroups $S_{\frac{n}{p_ip_a^s}}$ of $C_n$, where $i\in [r]\setminus\{a\}$. Then
$$\left|Q_a^s\right| = \frac{n}{p_1p_2\cdots p_r}\times\frac{1}{p_a^{s-1}}\times \left[\frac{p_1p_2\cdots p_r}{p_a} - \phi\left(\frac{p_1p_2\cdots p_r}{p_a}\right) \right].$$
\end{lemma}

\begin{proof}
Consider the subgroup $S_{\frac{n}{p_a^s}}$ of $C_n$ containing $Q_a^s$. Let $\widetilde{Q_a^s}$ denote the complement of $Q_a^s$ in $S_{\frac{n}{p_a^s}}$. If $x\in S_{\frac{n}{p_a^s}}$ and $o(x)=\frac{n}{p_a^k}$ for some $k$ with $s\leq k\leq n_a$, then $x\in\widetilde{Q_a^s}$. Conversely, if $y\in\widetilde{Q_a^s}$, then $o(y)$ must be divisible by $p_i^{n_i}$ for every $i\in [r]\setminus\{a\}$ and so $o(y)$ is of the form $\frac{n}{p_a^k}$, where $s\leq k\leq n_a$.

Thus $\widetilde{Q_a^s}=E_{\frac{n}{p_a^s}}\cup E_{\frac{n}{p_a^{s+1}}}\cup \cdots \cup E_{\frac{n}{p_a^{n_a}}}$, which is a disjoint union. Using Lemma \ref{lem-2.1}, we have
$
|\widetilde{Q_a^s}|=\underset{l=s}{\overset{n_a}{\sum}} \phi\left( \frac{n}{p_a^{l}}\right)= \frac{n}{p_1p_2\cdots p_r}\times \frac{1}{p_a^{s-1}}\times \phi\left(\frac{p_1p_2\cdots p_{r}}{p_a}\right)
$. Hence $|Q_a^s|=\left|S_{\frac{n}{p_a^s}} \right| - |\widetilde{Q_a^s}|= \frac{n}{p_1p_2\cdots p_r}\times\frac{1}{p_a^{s-1}}\times \left[\frac{p_1p_2\cdots p_r}{p_a} - \phi\left(\frac{p_1p_2\cdots p_r}{p_a}\right) \right].$
\end{proof}

\begin{lemma}\label{lem-Q-a-s-1}
Let $r\geq 3$. If  $a,b\in [r]$ with $a<b$, then $\left|Q_a^1\right| > \left|Q_b^1\right|$. As a consequence, $\left|Z_a^1 \right|> \left|Z_b^1 \right|$.
\end{lemma}

\begin{proof}
Using Lemma \ref{lem-Q-a-s}, we get $\left|Q_a^1\right|-\left|Q_b^1\right|= \frac{n}{p_1p_2\cdots p_r}\times (p_b -p_a) \left[\frac{p_1p_2\cdots p_r}{p_ap_b} -\phi\left(\frac{p_1p_2\cdots p_r}{p_ap_b}\right)\right]$. Since $a<b$, we have $p_a <p_b$. As $r\geq 3$, it follows that $\left|Q_a^1\right|>\left|Q_b^1\right|$. Then $\left|Z_a^1 \right|=|E_n|+ \left|Q_a^1\right|> |E_n|+ \left|Q_b^1\right| = \left|Z_b^1 \right|$.
\end{proof}

\begin{lemma}\label{lem-to-use}
Let $r\geq 3$ and $a,b\in [r]$ such that $a\neq b$ and $p_b > 2$. If $T$ is the union of the $r-2$ subgroups $S_{\frac{n}{p_bp_j}}$ of $C_n$ with $j\in [r]\setminus\{a,b\}$, then $\left|E_{\frac{n}{p_a}} \right| + |T| > \left|Q_r^1\right|$.
\end{lemma}

\begin{proof}
Note that the union of $T$ and $S_{\frac{n}{p_ap_b}}$ is $Q_b^1$. Let $\widetilde{T}$ be the complement of $T$ in $Q_b^1$. Then $\widetilde{T}$ is a subset of $S_{\frac{n}{p_ap_b}}$. Further, $\widetilde{T}$ is precisely union of the mutually disjoint sets $E_{\frac{n}{p_a^kp_b^l}}$, where $1\leq k\leq n_a$, $1\leq l\leq n_b$. By Lemma \ref{lem-2.4}, we get $|\widetilde{T}|=\phi\left( \frac{n}{p_a^{n_a}p_b^{n_b}}\right) p_a^{n_a-1} p_b^{n_b-1}=\frac{n}{p_1p_2\cdots p_r}\times\phi\left(\frac{p_1p_2\cdots p_r}{p_ap_b}\right)$.
We have
$$\left|E_{\frac{n}{p_a}} \right|=\phi\left(\frac{n}{p_a} \right)\geq \frac{n}{p_1p_2\cdots p_r} \times\frac{\phi(p_1p_2\cdots p_r)}{p_a}=|\widetilde{T}|\times \frac{\phi(p_ap_b)}{p_a}.$$
If $b=r$, then $p_b\geq 5$ as $r\geq 3$. In this case, $|\widetilde{T}|\times \frac{\phi(p_ap_b)}{p_a}> |\widetilde{T}|$ and so $\left|E_{\frac{n}{p_a}} \right| + |T| > |\widetilde{T}| +|T|=|Q_r^1|$. If $b <r$, then $|\widetilde{T}|\times \frac{\phi(p_ap_b)}{p_a}\geq |\widetilde{T}|$ as $p_b >2$ and so $\left|E_{\frac{n}{p_a}} \right| + |T|\geq |\widetilde{T}| +|T|=|Q_b^1| > \left|Q_r^1\right|$ by Lemma \ref{lem-Q-a-s-1}.
\end{proof}

\begin{lemma}\label{lem-Q-a-b-s}
Let $a,b\in [r]$ with $a\neq b$, $s\in [n_a]$ and let $Q_{a,b}^s$ be the union of the subgroups $S_{\frac{n}{p_ip_a^s}}$ of $C_n$, where $i\in [r]\setminus\{a,b\}$. Then
$$\left|Q_{a,b}^s\right| = \frac{n}{p_1p_2\cdots p_r}\times\frac{p_b}{p_a^{s-1}}\times \left[\frac{p_1p_2\cdots p_r}{p_ap_b} - \phi\left(\frac{p_1p_2\cdots p_r}{p_ap_b}\right) \right].$$
\end{lemma}

\begin{proof}
Consider the subgroup $S_{\frac{n}{p_a^s}}$ of $C_n$ containing $Q_{a,b}^s$. Let $\widetilde{Q_{a,b}^s}$ denote the complement of $Q_{a,b}^s$ in $S_{\frac{n}{p_a^s}}$. Then $\widetilde{Q_{a,b}^s}$ is a disjoint union of the sets $E_{\frac{n}{p_a^kp_b^l}}$, where $s\leq k\leq n_a$ and $0\leq l\leq n_b$. Using Lemma \ref{lem-2.4}, we get
\begin{align*}
|\widetilde{Q_{a,b}^s}| = \phi\left(\frac{n}{p_a^{n_a}p_b^{n_b}}\right) p_a^{n_a-s} p_b^{n_b}= \frac{n}{p_1p_2\cdots p_r}\times \frac{p_b}{p_a^{s-1}}\times \phi\left(\frac{p_1p_2\cdots p_r}{p_ap_b}\right).
\end{align*}
Hence $|Q_{a,b}^s|=\left|S_{\frac{n}{p_a^s}} \right| - |\widetilde{Q_{a,b}^s}|= \frac{n}{p_1p_2\cdots p_r}\times\frac{p_b}{p_a^{s-1}}\times \left[\frac{p_1p_2\cdots p_r}{p_ap_b} - \phi\left(\frac{p_1p_2\cdots p_r}{p_ap_b}\right) \right].$
\end{proof}

\begin{lemma}\label{lem-two-atmost}
$\phi\left(\frac{n}{p_j}\right) > \frac{n}{p_jp_r} -\phi\left( \frac{n}{p_jp_r} \right)$ for $j\in [r-1]$.
\end{lemma}

\begin{proof}
Let $\mu :=\phi\left(\frac{n}{p_j}\right) + \phi\left( \frac{n}{p_jp_r} \right)= \phi\left(\frac{n}{p_j^{n_j}p_r^{n_r}}\right) \phi\left(p_j^{n_j -1}\right) \left[\phi\left(p_r^{n_r}\right) + \phi\left(p_r^{n_r -1}\right) \right]$. We show that $\mu > \frac{n}{p_jp_r}$.
Since $\phi\left(\frac{n}{p_j^{n_j}p_r^{n_r}}\right)= \frac{n}{p_1p_2\cdots p_r}\times \frac{1}{p_j^{n_j -1}p_r^{n_r -1}}\times \phi\left(\frac{p_1p_2\cdots p_r}{p_jp_r}\right)$ and
\begin{align*}
\phi\left(p_j^{n_j -1}\right)  \left[\phi\left(p_r^{n_r}\right) + \phi\left(p_r^{n_r -1}\right) \right]
& = \begin{cases}
p_j^{n_j -2}p_r^{n_r -2}(p_r +1)\phi(p_jp_r), &\text{if $n_j\geq 2$, $n_r\geq 2$}\\
p_j^{n_j -2}p_r\phi(p_j), &\text{if $n_j\geq 2$, $n_r=1$}\\
p_r^{n_r -2}(p_r +1)\phi(p_r), &\text{if $n_j=1$, $n_r\geq 2$}\\
p_r, &\text{if $n_j=1= n_r$},
\end{cases}
\end{align*}
we get
\begin{align*}
\mu
& = \begin{cases}
 \frac{n}{p_1p_2\cdots p_r}\times (p_r +1) \times\frac{\phi\left(p_1p_2\cdots p_r\right)}{p_jp_r}, &\text{if $n_j\geq 2$, $n_r\geq 2$}\\
 \frac{n}{p_1p_2\cdots p_r}\times \frac{p_r}{p_j}\times\phi\left(\frac{p_1p_2\cdots p_r}{p_r}\right), &\text{if $n_j\geq 2$, $n_r=1$}\\
 \frac{n}{p_1p_2\cdots p_r}\times \frac{p_r +1}{p_r}\times\phi\left(\frac{p_1p_2\cdots p_r}{p_j}\right), &\text{if $n_j=1$, $n_r\geq 2$}\\
 \frac{n}{p_1p_2\cdots p_r}\times p_r\times\phi\left(\frac{p_1p_2\cdots p_r}{p_jp_r}\right), &\text{if $n_j=1= n_r$}
\end{cases}\\
& > \frac{n}{p_1p_2\cdots p_r}\times \frac{p_1p_2\cdots p_r}{p_jp_r} = \frac{n}{p_jp_r}.
\end{align*}
The strict inequality in the above follows from (\ref{eqn-2}) using the fact that $p_r \geq r+1$.
\end{proof}

\begin{lemma}\label{none-1}
Let $I_1$ be a nonempty proper subset of $[r]$, $a\in [r]\setminus I_1$ and $I_2=[r]\setminus (I_1\cup\{a\})$. If $K$ is the union of the subgroups $S_{\frac{n}{p_ip_a}}$ of $C_n$ for $i\in I_1$, then
$$|K|=\frac{n}{p_1p_2\cdots p_r}\times \underset{j\in I_2}\prod p_j \times \left[\underset{i\in I_1}\prod p_i - \phi\left(\underset{i\in I_1}\prod p_i\right)\right],$$
where $\underset{j\in I_2}\prod p_j =1$ if $I_2$ is the empty set.
\end{lemma}

\begin{proof}
Let $|I_1|=t$, $\alpha:=\underset{i\in I_1}\prod p_i$ and $\beta:=\underset{k\in I_2}\prod p_k$. Since $K=\underset{i\in I_1}{\bigcup} S_{\frac{n}{p_ip_a}}$, we get
\begin{align*}
\left\vert K\right\vert & = \underset{i\in I_1}{\sum} \left\vert S_{\frac{n}{p_ip_a}}\right\vert  - \underset{i<j}{\underset{i,j\in I_1}{\sum}} \left\vert S_{\frac{n}{p_ip_a}}\bigcap S_{\frac{n}{p_jp_a}}\right\vert + \cdots + (-1)^{t-1} \left\vert \underset{i\in I_1} {\bigcap} S_{\frac{n}{p_ip_a}}\right\vert \\
& = \underset{i\in I_1}{\sum} \frac{n}{p_ip_a} - \underset{i<j}{\underset{i,j\in I_1}{\sum}} \frac{n}{p_ip_jp_a} + \cdots + (-1)^{t-1} \frac{n}{\alpha p_a} \\
 & = \frac{n}{p_1p_2\cdots p_r}\times \beta\times \left[\underset{i\in I_1}{\sum}\frac{\alpha}{p_i} - \underset{i<j}{\underset{i,j\in I_1}{\sum}}\frac{\alpha}{p_ip_j} +\cdots +(-1)^{t-1}\right]\\
 & = \frac{n}{p_1p_2\cdots p_r}\times \beta \times [\alpha - \phi(\alpha)],
\end{align*}
where the last equality holds using (\ref{eqn-3}).
\end{proof}

\section{The cut-sets $Z_a^s$ and $X_{a,b}^{s,t}$}

For a nonempty proper subset $X$ of $C_n$, we denote by $\overline{X}$ the complement of $X$ in $C_n$ and by $\mathcal{P}(\overline{X})$ the induced subgraph of $\mathcal{P}(C_n)$ with vertex set $\overline{X}$.

\subsection{The cut-sets $Z_a^s$}\label{sec-Z-a-s}

Let $r\geq 2$, $a\in [r]$ and $s\in [n_a]$. Recall that $Z_a^s=E_n\cup E_{\frac{n}{p_a}}\cup E_{\frac{n}{p_a^2}}\cup\ldots\cup E_{\frac{n}{p_a^{s-1}}}\cup Q_a^s$, where $Q_a^s$ is the union of the subgroups $S_{\frac{n}{p_ip_a^s}}$ of $C_n$ for $i\in [r]\setminus\{a\}$.

\begin{proposition}\label{prop-Z-a-s}
For $r\geq 2$, $Z_a^s$ is a cut-set of $\mathcal{P}(C_n)$.
\end{proposition}

\begin{proof}
Take $A=E_{\frac{n}{p_a^{s}}}\cup E_{\frac{n}{p_a^{s+1}}}\cup\ldots\cup E_{\frac{n}{p_a^{n_a}}}$ and $B=C_n\setminus (Z_a^s\cup A)$. Then $A,B$ are nonempty disjoint subsets of $\overline{Z_a^s}$ and $A\cup B=\overline{Z_a^s}$.

For $x\in \overline{Z_a^s}$, the order of $x$ is of the form $p_1^{l_1}\cdots p_a^{l_a}\cdots p_r^{l_r}$, where $0\leq l_i\leq n_i$ for $i\in [r]$. Note that $x\in A$ if and only if $l_i=n_i$ for every $i\in [r]\setminus\{a\}$ and $l_a\in\{0,1,\dots,n_a -s\}$. On the other hand, $x\in B$ if and only if $l_i\neq n_i$ for at least one $i\in [r]\setminus\{a\}$ and $l_a\in\{n_a- s+1,\ldots, n_a\}$. It follows that the order of an element of $A$ does not divide the order of any element of $B$ and vice versa. So there is no edge of $\mathcal{P}(C_n)$ containing vertices from both $A$ and $B$.
Thus $A\cup B$ is a separation of $\mathcal{P}(\overline{Z_a^s})$ and hence $Z_a^s$ is a cut-set of $\mathcal{P}(C_n)$.
\end{proof}

\noindent The following is a consequence of Proposition \ref{prop-Z-a-s}.

\begin{corollary}\label{cor-Z-a-s}
If $r\geq 2$ and $X$ is a minimum cut-set of $\mathcal{P}(C_n)$, then $|X|\leq |Z_a^s|$ for every $a\in [r]$ and $s\in [n_a]$.
\end{corollary}

\begin{proposition}\label{size-Z-a-s}
$|Z_a^s| =\phi(n) + \frac{n}{p_1p_2\cdots p_r}\times \frac{1}{p_a^{s-1}}\times \left[\frac{p_1p_2\cdots p_{r}}{p_a} +\phi\left(\frac{p_1p_2\cdots p_{r}}{p_a}\right)\left(p_a^{s-1}-2\right) \right]$.
\end{proposition}

\begin{proof}
Note that $Z_a^s$ is a disjoint union of the sets $E_n, E_{\frac{n}{p_a}}, E_{\frac{n}{p_a^2}}, \ldots, E_{\frac{n}{p_a^{s-1}}}$ and $Q_a^s$. We have $|E_n|=\phi(n)$. By Lemmas \ref{lem-2.1} and \ref{lem-Q-a-s}, we get
$$\underset{l=1}{\overset{s-1}{\sum}} \left\vert E_{\frac{n}{p_a^{l}}} \right\vert = \frac{n}{p_1p_2\cdots p_r}\times \frac{1}{p_a^{s-1}}\times \phi\left(\frac{p_1\cdots p_r}{p_a} \right) \left(p_a^{s -1}-1 \right)$$
and
$$|Q_a^s| = \frac{n}{p_1p_2\cdots p_r}\times\frac{1}{p_a^{s-1}}\times \left[\frac{p_1\cdots p_r}{p_a} - \phi\left(\frac{p_1\cdots p_r}{p_a}\right) \right].$$
It then follows that $|Z_a^s| =\phi(n) + \frac{n}{p_1p_2\cdots p_r}\times \frac{1}{p_a^{s-1}}\times \left[\frac{p_1p_2\cdots p_{r}}{p_a} +\phi\left(\frac{p_1p_2\cdots p_{r}}{p_a}\right)\left(p_a^{s-1}-2\right) \right]$.
\end{proof}

\begin{proposition}\label{prop-comp-Z-a-s}
Let $r\geq 3$, $a\in [r]$ and $n_a\geq 2$. Then the following hold:
\begin{enumerate}
\item[(i)] If $2\phi\left(\frac{p_1p_2\cdots p_r}{p_a}\right)>\frac{p_1p_2\cdots p_r}{p_a}$, then $\left| Z_a^1\right| <\left| Z_a^2\right| <\cdots < \left| Z_a^{n_a}\right|$.
\item[(ii)] If $2\phi\left(\frac{p_1p_2\cdots p_r}{p_a}\right)<\frac{p_1p_2\cdots p_r}{p_a}$, then $\left| Z_a^1\right|>\left| Z_a^2\right|>\cdots > \left| Z_a^{n_a}\right|$.
\end{enumerate}
\end{proposition}

\begin{proof}
Since $n_a\geq 2$, we have $n_a-1\geq 1$. For $s\in \{1,2,\ldots,n_a -1\}$, it can be calculated using Proposition \ref{size-Z-a-s} that
\begin{align*}
\left| Z_a^s\right| -\left| Z_a^{s+1}\right|&= \frac{n}{p_1\cdots p_r} \times \left(\frac{1}{p_a^{s-1}}-\frac{1}{p_a^{s}}\right)\left[ \frac{p_1p_2\cdots p_r}{p_a} - 2\phi\left(\frac{p_1p_2\cdots p_r}{p_a}\right)\right].
\end{align*}
Since $r\geq 3$, we have $2\phi\left(\frac{p_1p_2\cdots p_r}{p_a}\right)\neq \frac{p_1p_2\cdots p_r}{p_a}$ and so the proposition follows.
\end{proof}

\subsection{The cut-sets $X_{a,b}^{s,t}$}\label{sec-X-a-b-s-t}

Let $r\geq 3$, $a,b\in [r]$ with $a\neq b$, $s\in [n_a]$ and $t\in [n_b]$. Recall that $X_{a,b}^{s,t}=H_{a,b}^{s,t}\bigcup K_{a,b}^{s,t}$, where $K_{a,b}^{s,t}$ is the set of all the nongenerators of the cyclic subgroup $S_{\frac{n}{p_a^s p_b^t}}$ of $C_n$ and $H_{a,b}^{s,t}$ is the union of the mutually disjoint sets $E_{\frac{n}{p_a^ip_b^j}}$ with $0\leq i\leq s$, $0\leq j\leq t$ and $(i,j)\neq (s,t)$. The set $E_n$ corresponds to $(i,j)=(0,0)$.

\begin{proposition}\label{prop-X-a-b-s-t}
For $r\geq 3$, $X_{a,b}^{s,t}$ is a cut-set of $\mathcal{P}(C_n)$.
\end{proposition}

\begin{proof}
Note that for $x\in C_n$, we have $x\in H_{a,b}^{s,t}$ if and only if $o(x)\neq\frac{n}{p_a^s p_b^t}$ and $o(x)$ is divisible by $\frac{n}{p_a^s p_b^t}$. Further, $x\in K_{a,b}^{s,t}$ if and only if $o(x)\neq\frac{n}{p_a^s p_b^t}$ and $o(x)$ divides $\frac{n}{p_a^s p_b^t}$. Thus $H_{a,b}^{s,t}$ and $K_{a,b}^{s,t}$ are disjoint sets and $X_{a,b}^{s,t}$ is the set of all elements of $C_n\setminus E_{\frac{n}{p_a^s p_b^t}}$ which are adjacent with some (and hence all) elements of $E_{\frac{n}{p_a^s p_b^t}}$. Taking $A= E_{\frac{n}{p_a^s p_b^t}}$ and $B = C_n\setminus \left(A\cup X_{a,b}^{s,t}\right)$, it then follows that $X_{a,b}^{s,t}$ is a cut-set of $\mathcal{P}(C_n)$ with $A\cup B$ a separation of $\mathcal{P}\left(\overline{X_{a,b}^{s,t}}\right)$.
\end{proof}

\noindent The following is a consequence of Proposition \ref{prop-X-a-b-s-t}.

\begin{corollary}\label{cor-X-a-b-s-t}
If $r\geq 3$ and $X$ is a minimum cut-set of $\mathcal{P}(C_n)$, then $|X|\leq \left|X_{a,b}^{s,t}\right|$ for every $a,b\in [r]$ with $a\neq b$, $s\in [n_a]$ and $t\in [n_b]$.
\end{corollary}

\section{Distribution of the sets $E_{\frac{n}{p_i}},$ $i\in [r]$}\label{distribution}

We start with two results (Propositions \ref{useful} and \ref{dist}) which will be used frequently in the rest of this paper mostly without mention. The first one was proved in \cite[Lemma 2.3]{cps}.

\begin{proposition}[\cite{cps}]\label{dist}
Let $Y$ be a minimum cut-set of $\mathcal{P}(C_n)$. If $d$ is a positive divisor of $n$, then either $E_d$ is a subset of $Y$ or $E_d$ is disjoint from $Y$.
\end{proposition}

Recall that two distinct elements $x,y$ of $C_n$ are adjacent in $\mathcal{P}(C_n)$ if and only if $o(x)\mid o(y)$ or $o(y)\mid o(x)$. As a consequence, we have the following:
\begin{proposition}\label{useful}
Suppose that $Y$ is a cut-set of $\mathcal{P}(C_n)$ and $A\cup B$ is a separation of $\mathcal{P}(\overline{Y})$ such that $A$ contains an element of order $a$ and $B$ contains an element of order $b$. If $d$ is the greatest common divisor of $a$ and $b$, then $Y$ contains the unique subgroup $S_d$ of $C_n$.
\end{proposition}

Now let $X$ be a minimum cut-set of $\mathcal{P}(C_n)$ and $A\cup B$ be a separation of $\mathcal{P}(\overline{X})$. Then for every positive divisor $d$ of $n$, there are three possibilities for the set $E_d$: either $E_d\subseteq X$, $E_d\subseteq A$ or $E_d\subseteq B$. This follows as a consequence of Proposition \ref{dist}. We study the distribution of the sets $E_{\frac{n}{p_i}}$, $i\in [r]$, among $A,B$ and $X$.
Define three subsets $U,V, W$ of $[r]$ associated with $A,B,X$ respectively by
$$U:=\left\{i\in [r]: E_{\frac{n}{p_i}}\subseteq A\right\},\; V:=\left\{i\in [r]: E_{\frac{n}{p_i}}\subseteq B\right\},\; W:= \left\{i\in [r]: E_{\frac{n}{p_i}}\subseteq X\right\}.$$
Then $U$, $V$ and $W$ are pairwise disjoint and $|U|+|V|+|W|=r$. We prove the following result using the cut-set $X_{a,b}^{1,1}$ of $\mathcal{P}(C_n)$ for suitable $a,b\in [r]$.

\begin{proposition}\label{prop-two-atmost}
$|W|\in\{0,1,2\}$.
\end{proposition}

\begin{proof}
Suppose that $|W|\geq 3$. Then $X$ contains $E_{\frac{n}{p_j}}, E_{\frac{n}{p_k}}$ and $E_{\frac{n}{p_l}}$ for some $j,k,l\in [r]$ with $j<k<l$. Since $X$ contains the set $E_n$ also, we get
\begin{align*}
|X| & \geq |E_n|+\left| E_{\frac{n}{p_j}}\right| + \left| E_{\frac{n}{p_k}}\right| + \left| E_{\frac{n}{p_l}}\right|\geq |E_n|+\left| E_{\frac{n}{p_k}}\right| + \left| E_{\frac{n}{p_k}}\right| + \left| E_{\frac{n}{p_r}}\right|
\end{align*}
using (\ref{eqn-1}). The cut-set $X_{k,r}^{1,1}$ of $\mathcal{P}(C_n)$ is a disjoint union of $E_n$, $E_{\frac{n}{p_k}}$, $E_{\frac{n}{p_r}}$ and $K_{k,r}^{1,1}$. We have $\left|K_{k,r}^{1,1}\right|=\frac{n}{p_kp_r} - \phi\left(\frac{n}{p_kp_r}\right)$, as $K_{k,r}^{1,1}$ is the set of all the nongenerators of the cyclic subgroup $S_{\frac{n}{p_kp_r}}$ of $C_n$. By Lemma \ref{lem-two-atmost}, $\left| E_{\frac{n}{p_k}}\right|=\phi\left(\frac{n}{p_k}\right) > \frac{n}{p_kp_r} - \phi\left(\frac{n}{p_kp_r}\right) = \left|K_{k,r}^{1,1}\right|$. It then follows that $|X| > |E_n|+ \left| E_{\frac{n}{p_k}}\right| + \left| E_{\frac{n}{p_r}}\right| + \left|K_{k,r}^{1,1}\right| = \left|X_{k,r}^{1,1}\right|$, a contradiction to Corollary \ref{cor-X-a-b-s-t}.
\end{proof}

We note that Proposition \ref{prop-two-atmost} was also obtained in \cite[Proposition 3.10]{cps-2} using a different cut-set of $\mathcal{P}(C_n)$. The following proposition is proved by implicitly borrowing some of the arguments used in the proofs of \cite[Propositions 4.2, 4.3]{cps-2}.

\begin{proposition}\label{U-V-W}
If both $U$ and $V$ are nonempty, then $W$ is the empty set and one of $U, V$ must be the singleton set $\{r\}$.
\end{proposition}

\begin{proof}
Fix $s\in U$ and $t\in V$. Without loss of generality, we may assume that $p_s>p_t$. Then $p_s >2$. We show that $s=r$, $U=\{r\}$ and $W=\emptyset$.

Let $D$ be the union of the subgroups $S_{\frac{n}{p_s p_l}}$ of $C_n$, where $l\in V$.
Since $A\cup B$ is a separation of $\mathcal{P}(\overline{X})$, the subgroups involved in $D$ must be contained in $X$ and so $D$ is a subset of $X$. By Proposition \ref{prop-two-atmost}, we have $|W|\leq 2$. We complete the proof in four steps.\medskip

{\bf (I)}: We claim that $|W|\neq 2$. Suppose $W=\{a,b\}$ for some $a,b\in [r]\setminus\{s,t\}$ with $a < b$. Then $a\leq r-1$ and $b\leq r$. So
$\left|E_{\frac{n}{p_a}}\right|\geq \left|E_{\frac{n}{p_{r-1}}}\right|$ and $\left|E_{\frac{n}{p_b}}\right| \geq \left|E_{\frac{n}{p_r}}\right|$ by (\ref{eqn-1}). As $D$ contains $S_{\frac{n}{p_s p_t}}$, we have
$|D|\geq \left|S_{\frac{n}{p_s p_t}}\right|=\frac{n}{p_sp_t}\geq \frac{n}{p_{r-1}p_r}=\left|S_{\frac{n}{p_{r-1} p_r}}\right|>\left|K_{r-1,r}^{1,1}\right|.$
Since $X$ contains the mutually disjoint sets $E_n$, $E_{\frac{n}{p_a}}$, $E_{\frac{n}{p_b}}$ and $D$, we get
\begin{align*}
|X| & \geq |E_n|+\left|E_{\frac{n}{p_a}}\right|+\left|E_{\frac{n}{p_b}}\right|+|D| > |E_n|+\left|E_{\frac{n}{p_{r-1}}}\right|+\left|E_{\frac{n}{p_r}}\right|+\left|K_{r-1,r}^{1,1}\right|.
\end{align*}
The cut-set $X_{r-1,r}^{1,1}$ of $\mathcal{P}(C_n)$ is the union of mutually disjoint sets $E_n$, $E_{\frac{n}{p_{r-1}}}$, $E_{\frac{n}{p_r}}$ and $K_{r-1,r}^{1,1}$. It then follows that $|X|> \left|X_{r-1,r}^{1,1}\right|$, a contradiction to Corollary \ref{cor-X-a-b-s-t}.\medskip

{\bf (II)}: We claim that $(|U|, |W|)\neq (1,1)$. Otherwise, $U=\{s\}$ and $W=\{a\}$ for some $a\in [r]\setminus\{s,t\}$. Then $V=[r]\setminus\{s,a\}$ and so $D$ is the union of the $r-2$ subgroups $S_{\frac{n}{p_sp_j}}$, $j\in [r]\setminus\{s,a\}$, of $C_n$. Since $p_s >2$, we have $\left|E_{\frac{n}{p_a}} \right| + |D| > \left|Q_r^1\right|$ by Lemma \ref{lem-to-use}.
Since $X$ contains the mutually disjoint sets $E_n$, $E_{\frac{n}{p_a}}$ and $D$, we get
$|X| \geq |E_n|+ \left|E_{\frac{n}{p_a}} \right| + |D| >|E_n|+ |Q_r^1|=\left|Z_r^1\right|$, a contradiction to Corollary \ref{cor-Z-a-s}.\medskip

{\bf (III)}: We claim that $|U|=1$. Suppose this is not the case. Then $|U|\geq 2$. Define the sets $D_1$ and $D_2$ by
$$D_1=\underset{k\in U\setminus\{s\}}{\bigcup} S_{\frac{n}{p_k p_t}},\;\; D_2=\underset{k\in U\setminus\{s\}}{\bigcup} S_{\frac{n}{p_sp_k}}.$$
Note that $D_1$ and $D_2$ are nonempty sets as $|U|\geq 2$. Since $A\cup B$ is a separation of $\mathcal{P}(\overline{X})$, the subgroups involved in $D_1$ must be contained in $X$. Thus $D\cup D_1$ is a subset of $X$.
We assert that $|D\cup D_1|> |D\cup D_2|$.

If $x\in D\cap D_1$, then $x\in S_{\frac{n}{p_s p_l}}\bigcap S_{\frac{n}{p_k p_t}}$ for some $l\in V$ and $k\in U\setminus\{s\}$. Since $S_{\frac{n}{p_s p_l}}\bigcap S_{\frac{n}{p_k p_t}}$ is a subgroup of $S_{\frac{n}{p_s p_kp_t}} = S_{\frac{n}{p_s p_t}}\bigcap S_{\frac{n}{p_s p_k}}$, we get $x\in D\cap D_2$. Thus $D\cap D_1$ is a subset of $D\cap D_2$ and hence $|D\cap D_1|\leq |D\cap D_2|$.

Put $R_1=[r]\setminus ((U\setminus\{s\})\cup \{t\})$ and $R_2=[r]\setminus U=[r]\setminus ((U\setminus\{s\})\cup \{s\})$. Observe that $t\notin R_1$, $t\in R_2$, $s\in R_1$, $s\notin R_2$ and $R_1\setminus\{s\}=R_2\setminus\{t\}$. Since $p_s >p_t$ by our assumption, we have $\underset{j\in R_1}\prod p_j > \underset{j\in R_2}\prod p_j$. Then using Lemma \ref{none-1}, we get
\begin{align*}
|D_1| & =\frac{n}{p_1p_2\cdots p_r}\times \underset{j\in R_1}\prod p_j \times \left[\underset{k\in U\setminus\{s\}}\prod p_k - \phi\left(\underset{k\in U\setminus\{s\}}\prod p_k\right)\right]\\
 & > \frac{n}{p_1p_2\cdots p_r}\times \underset{j\in R_2}\prod p_j \times \left[\underset{k\in U\setminus\{s\}}\prod p_k - \phi\left(\underset{k\in U\setminus\{s\}}\prod p_k\right)\right]=|D_2|.
\end{align*}
Therefore, $|D\cup D_1|=|D|+|D_1|-|D\cap D_1|> |D|+|D_2|-|D\cap D_2|= |D\cup D_2|$.

If $W=\emptyset$, then $U\cup V=[r]$ and so $D\cup D_2=Q_s^1$. Since $X$ contains $E_n$ and $D\cup D_1$, we get $|X|\geq |E_n|+ |D\cup D_1|>|E_n|+ |D\cup D_2|= |E_n|+ |Q_s^1|=\left|Z_s^1\right|$, a contradiction to Corollary \ref{cor-Z-a-s}.

If $|W|=1$, then $W=\{a\}$ for some $a\in [r]\setminus\{s,t\}$. In this case, $U\cup V=[r]\setminus\{a\}$ and so $D\cup D_2$ is the union of the $r-2$ subgroups $S_{\frac{n}{p_sp_j}}$, $j\in [r]\setminus\{s,a\}$, of $C_n$. As $p_s >2$, we have $\left|E_{\frac{n}{p_a}} \right| + |D\cup D_2| > \left|Q_r^1\right|$ by Lemma \ref{lem-to-use}.
Since $X$ contains the mutually disjoint sets $E_n$, $E_{\frac{n}{p_a}}$ and $D\cup D_1$, we get
$|X| \geq |E_n|+ \left|E_{\frac{n}{p_a}} \right| + |D\cup D_1|  >|E_n|+ \left|E_{\frac{n}{p_a}} \right| +|D\cup D_2|>|E_n| + |Q_r^1|=\left|Z_r^1\right|$, a contradiction to Corollary \ref{cor-Z-a-s}.\medskip

{\bf (IV)}: We claim that $s=r$. By the previous three steps, we get $U=\{s\}$ and $W=\emptyset$. So $V=[r]\setminus\{s\}$ and hence $D=Q_s^1$. If $s\neq r$, then using Lemma \ref{lem-Q-a-s-1}, we get $|X|\geq |E_n|+ |D|=|E_n|+ |Q_s^1|=\left|Z_s^1\right|>\left|Z_r^1\right|$, a contradiction to Corollary \ref{cor-Z-a-s}.
\end{proof}

It follows from the proof of \cite[Proposition 4.6]{cps-2} that if $n_r\geq 2$ and $|W|= 2$, then $|X|>|Z_r^{n_r}|$, contradicting Corollary \ref{cor-Z-a-s}. We thus have the following:

\begin{proposition}\label{nr-1}
If $|W|=2$, then $n_r = 1$.
\end{proposition}

\section{Proof of the Main Theorem}\label{sec-main-proof}

Let $r\geq 4$ and $X$ be a minimum cut-set of $\mathcal{P}(C_n)$. Fix a separation $A\cup B$ of $\mathcal{P}(\overline{X})$. We continue with the sets $U,V$ and $W$ as defined in Section \ref{distribution} which are associated with $A,B$ and $X$ respectively. By Proposition \ref{prop-two-atmost}, we have $|W|\leq 2$.

\subsection{The case $|W|=0$}

Suppose that $W=\emptyset$. Then $U\neq \emptyset$ and $V\neq \emptyset$, as both $A$ and $B$ are nonempty subsets of $\overline{X}$ and $E_n$ is contained in $X$. So $\{U, V\}=\{\{r\},[r-1]\}$ by Proposition \ref{U-V-W}. It then follows that $X$ contains the subgroups $S_{\frac{n}{p_ip_r}}$ for $i\in[r-1]$ and hence contains the cut-set $Z_r^1$ of $\mathcal{P}(C_n)$. Since $X$ is a minimum cut-set of $\mathcal{P}(C_n)$, we must have $X=Z_r^1$. Thus we have the following:

\begin{proposition}\label{exactly-one}
If $W=\emptyset$, then $X=Z_r^1$.
\end{proposition}

\subsection{The case $|W|=1$}

\begin{proposition}\label{exactly-one}
If $W=\{a\}$ for some $a\in [r]$, then $n_a\geq 2$, $2\phi\left(\frac{p_1p_2\cdots p_r}{p_a}\right) < \frac{p_1p_2\cdots p_r}{p_a}$ and $X=Z_a^{n_a}$.
\end{proposition}

\begin{proof}
Without loss of generality, we may assume by Proposition \ref{U-V-W} that $U=\emptyset$. Then $V=[r]\setminus \{a\}$. Since $B$ contains the sets $E_{\frac{n}{p_i}}$ for $i\in [r]\setminus \{a\}$ and there is no edge of $\mathcal{P}(C_n)$ containing vertices from both $A$ and $B$, the order of an element in $A$ must be divisible by $p_i^{n_i}$ for every $i\in [r]\setminus \{a\}$ and hence must be of the form $\frac{n}{p_a^k}$, where $0\leq k\leq n_a$. The elements of $C_n$ of order $n$ and $\frac{n}{p_a}$ correspond to the elements contained in $E_n$ and $E_{\frac{n}{p_a}}$ respectively. Since $E_n$ and $E_{\frac{n}{p_a}}$ are contained in $X$ and $A$ is nonempty, it follows that $k\geq 2$ and hence $n_a\geq 2$.

Let $s\in\{2,3,\ldots,n_a\}$ be the smallest integer such that $A$ has an element of order $\frac{n}{p_a^s}$. Then the sets $E_{\frac{n}{p_a^2}},\dots,E_{\frac{n}{p_a^{s-1}}}$ and the subgroups $S_{\frac{n}{p_ip_a^s}}$, $i\in [r]\setminus\{a\}$, must be contained in $X$. It follows that $X$ contains the cut-set $Z_a^s$ of $\mathcal{P}(C_n)$. Since $X$ is a minimum cut-set of $\mathcal{P}(C_n)$, we must have $X=Z_a^s$.

Since $r\geq 4$, we have $2\phi\left(\frac{p_1p_2\cdots p_r}{p_a}\right)\neq \frac{p_1p_2\cdots p_r}{p_a}$. If $2\phi\left(\frac{p_1p_2\cdots p_r}{p_a}\right) > \frac{p_1p_2\cdots p_r}{p_a}$, then $s\geq 2$ implies $\left|Z_a^s \right|> \left|Z_a^1 \right|$ by Proposition \ref{prop-comp-Z-a-s}. This gives $|X|>\left|Z_a^1 \right|$, contradicting Corollary \ref{cor-Z-a-s}.

Thus $2\phi\left(\frac{p_1p_2\cdots p_r}{p_a}\right)< \frac{p_1p_2\cdots p_r}{p_a} $. In this case, by Proposition \ref{prop-comp-Z-a-s}, $\left|Z_a^s \right|\geq \left|Z_a^{n_a} \right|$ and equality holds if and only if $s=n_a$. Since $X=Z_a^s$ is a minimum cut-set of $\mathcal{P}(C_n)$, it follows that $s=n_a$ and $X=Z_a^{n_a}$.
\end{proof}

\subsection{The case $|W|=2$}

In the rest of this section, we consider that $W=\{a,b\}$ for some $a,b\in [r]$ with $a\neq b$.
Without loss of generality, we may assume by Proposition \ref{U-V-W} that $U=\emptyset$. Then $V=[r]\setminus \{a,b\}$ and so $B$ contains the sets $E_{\frac{n}{p_i}}$ for $i\in [r]\setminus \{a,b\}$.

\begin{proposition}\label{two-2-1}
If $x\in A$, then $o(x)$ cannot be of the form $\frac{n}{p_c^k}$, where $c\in\{a,b\}$ and $k\in [n_c]$.
\end{proposition}

\begin{proof}
Suppose that $A$ contains an element of order of the form $\frac{n}{p_c^k}$ for some $c\in\{a,b\}$ and $k\in [n_c]$. Without loss of generality, we may assume that $c=a$. Then $k\geq 2$ as $E_{\frac{n}{p_a}}$ is contained in $X$.

Let $s\in\{2,3,\ldots,n_a\}$ be the smallest integer such that $A$ contains an element of order $\frac{n}{p_a^s}$. Then the sets $E_{\frac{n}{p_a}},\ldots, E_{\frac{n}{p_a^{s-1}}}$ and the subgroups $S_{\frac{n}{p_ip_a^s}}$, $i\in [r]\setminus\{a,b\}$, must be contained in $X$. Let $Q_{a,b}^s$ denote the union of the $r-2$ subgroups $S_{\frac{n}{p_ip_a^s}}$, $i\in [r]\setminus\{a,b\}$. Since $Q_{a,b}^s$ is a subset of $Q_a^s$, we get
$$\left|Q_a^s\setminus Q_{a,b}^s\right|=\left|Q_a^s\right|-\left|Q_{a,b}^s\right|=\frac{n}{p_1p_2\cdots p_r}\times\frac{1}{p_a^{s-1}}\times \phi\left(\frac{p_1p_2\cdots p_r}{p_ap_b}\right)$$
using Lemmas \ref{lem-Q-a-s} and \ref{lem-Q-a-b-s}. As $s\geq 2$, we have $\frac{\phi(p_ap_b)}{p_b}>\frac{1}{p_a^{s-1}}$ and so
$$\left|E_{\frac{n}{p_b}}\right|= \phi\left(\frac{n}{p_b}\right)\geq \frac{n}{p_1\cdots p_r}\times \frac{\phi(p_1\cdots p_{r})}{p_b} > \frac{n}{p_1\cdots p_r}\times \frac{1}{p_a^{s-1}}\times \phi\left(\frac{p_1\cdots p_r}{p_ap_b}\right)=\left|Q_a^s\right|-\left|Q_{a,b}^s\right|.$$
This gives $\left|E_{\frac{n}{p_b}}\right| + \left|Q_{a,b}^s\right| > \left|Q_a^s\right|.$ The sets $E_n$, $E_{\frac{n}{p_a}}, \ldots, E_{\frac{n}{p_a^{s-1}}}$, $E_{\frac{n}{p_b}}$ and $Q_{a,b}^s$ are mutually disjoint and contained in $X$. So
\begin{align*}
|X|&\geq |E_n|+\left|E_{\frac{n}{p_a}}\right|+\cdots +\left|E_{\frac{n}{p_a^{s-1}}}\right|+\left|E_{\frac{n}{p_b}}\right|+\left|Q_{a,b}^s\right|\\
 & > |E_n|+\left|E_{\frac{n}{p_a}}\right|+\cdots +\left|E_{\frac{n}{p_a^{s-1}}}\right|+\left|Q_{a}^s\right|.
\end{align*}
Since the cut-set $Z_a^s$ of $\mathcal{P}(C_n)$ is the union of mutually disjoint sets $E_n$, $E_{\frac{n}{p_a}}$, $E_{\frac{n}{p_a^2}},\ldots,E_{\frac{n}{p_a^{s-1}}}$ and $Q_a^s$, it follows that
$|X| > \left|Z_a^s\right|$, a contradiction to Corollary \ref{cor-Z-a-s}.
\end{proof}

\begin{corollary}\label{elements-in-A}
If $x\in A$, then $o(x)$ is of the form $\frac{n}{p_a^{u}p_b^{v}}$ for some $u,v$ with $1\leq u\leq n_a$ and $1\leq v\leq n_b$.
\end{corollary}

\begin{proof}
By Proposition \ref{two-2-1}, $A$ has no element whose order is of the form $\frac{n}{p_c^k}$, where $c\in\{a,b\}$ and $k\in [n_c]$. Since $B$ contains the sets $E_{\frac{n}{p_i}}$ for $i\in [r]\setminus \{a,b\}$ and there is no edge of $\mathcal{P}(C_n)$ containing vertices from both $A$ and $B$, $o(x)$ must be divisible by $p_i^{n_i}$ for every $i\in [r]\setminus \{a,b\}$. It then follows that $o(x)$ is of the form $\frac{n}{p_a^{u}p_b^{v}}$ with $1\leq u\leq n_a$ and $1\leq v\leq n_b$.
\end{proof}

{\it Now on we assume that $a<b$}. Then $p_a<p_b$ and so $p_b\geq 3$. By Corollary \ref{elements-in-A}, let $t\in [n_b]$ be the smallest integer such that $A$ contains $E_{\frac{n}{p_a^up_b^t}}$ for some $u\in [n_a]$. For such $t$, let $s\in[n_a]$ be the smallest integer for which $E_{\frac{n}{p_a^sp_b^t}}$ is contained in $A$. The rest of this section is devoted to prove that $X=X_{a,b}^{s,t}$. 

Since the sets $E_{\frac{n}{p_i}}$, $i\in [r]\setminus\{a,b\}$, are contained in $B$ and there is no edge of $\mathcal{P}(C_n)$ containing vertices from both $A$ and $B$, it follows that the sets $E_{\frac{n}{p_a^ip_b^j}}$, where $0\leq i\leq s$, $0\leq j\leq t$ with $(i,j)\neq (s,t)$, and the subgroups $S_{\frac{n}{p_ip_a^sp_b^t}}$ of $\mathcal{P}(C_n)$, where $i\in [r]\setminus\{a,b\}$, must be contained in $X$. Thus we have the following:

\begin{proposition}\label{elements-in-X}
$X$ contains $H_{a,b}^{s,t}$ and the subgroups $S_{\frac{n}{p_ip_a^sp_b^t}}$ of $C_n$, where $i\in [r]\setminus\{a,b\}$.
\end{proposition}

\begin{corollary}\label{s=n-a t=n-b}
If $(s,t)=(n_a, n_b)$, then $X=X_{a,b}^{s,t}=X_{a,b}^{n_a,n_b}$. 
\end{corollary}

\begin{proof}
Since $(s,t)=(n_a, n_b)$, Proposition \ref{elements-in-X} implies that $X$ contains the set $K_{a,b}^{n_a,n_b}$ and hence the cut-set $X_{a,b}^{n_a,n_b}$ of $\mathcal{P}(C_n)$. Then the corollary follows by the minimality of $X$.
\end{proof}

By Corollary \ref{s=n-a t=n-b}, we shall assume that $(s,t)\neq (n_a, n_b)$. Our aim is to show that $X$ contains the subgroup $S_{\frac{n}{p_a^{s+1}p_b^t}}$ if $s<n_a$ and the subgroup $S_{\frac{n}{p_a^sp_b^{t+1}}}$ if $t<n_b$. For this, it is enough to show that $B$ contains the set $E_{\frac{n}{p_a^{s+1}}}$ if $s<n_a$ and the set $E_{\frac{n}{p_b^{t+1}}}$ if $t<n_b$. 

\begin{proposition}\label{not-(2,1)}
If $t<n_b$, then $(p_a, n_a)\neq (2,1)$.
\end{proposition}

\begin{proof}
Suppose that $(p_a,n_a)=(2,1)$. Then $a=1$ and $s=n_1=1$. We have $n_r=1$ by Proposition \ref{nr-1}. So $1<b < r$ (as $n_b>t\geq 1$) and hence $2<p_b <p_r$. We show that $|X| > \left|X_{r,b}^{1,t}\right|$, which would contradict Corollary \ref{cor-X-a-b-s-t}.

The cut-set $X_{r,b}^{1,t}$ of $\mathcal{P}(C_n)$ is a disjoint union of $H_{r,b}^{1,t}$ and $K_{r,b}^{1,t}$. By Proposition \ref{elements-in-X}, $X$ contains $H_{1,b}^{1,t}$ and the subgroups $S_{\frac{n}{2p_b^tp_i}}$ of $C_n$, where $i\in [r]\setminus\{1,b\}$. We know that $H_{r,b}^{1,t}$ is a disjoint union of the sets $E_{\frac{n}{p_r^i p_b^j}}$ and $H_{1,b}^{1,t}$ is a disjoint union of the sets $E_{\frac{n}{2^i p_b^j}}$, where $0\leq i\leq 1$, $0\leq j\leq t$ and $(i,j)\neq (1,t)$. For such $i$ and $j$, $p_r >2$ implies that $\left|E_{\frac{n}{2^i p_b^j}}\right| = \phi\left(\frac{n}{2^i p_b^j}\right)\geq \phi\left(\frac{n}{p_r^i p_b^j}\right)=\left|E_{\frac{n}{p_r^i p_b^j}}\right|$. It then follows that $\left| H_{1,b}^{1,t}\right| \geq \left| H_{r,b}^{1,t}\right|$.

Let $G$ be union of the subgroups $S_{\frac{n}{2p_b^t p_i}}$ of $C_n$, where $i\in [r]\setminus\{1,b\}$. Since each of these subgroups is disjoint from $H_{1,b}^{1,t}$ and $K_{r,b}^{1,t}$ is disjoint from $H_{r,b}^{1,t}$, it is enough to show that $|G| > \left| K_{r,b}^{1,t}\right|$. Since the set $K_{r,b}^{1,t}$ is the set of nongenerators of the cyclic subgroup $S_{\frac{n}{p_r p_b^t}}$ of $C_n$, we have
$$\left| K_{r,b}^{1,t}\right|=\frac{n}{p_rp_b^{t}}-\phi\left(\frac{n}{p_rp_b^{t}}\right)=\frac{n}{2p_2\cdots p_r}\times \frac{1}{p_b^{t}}\times \left[2p_2\cdots p_{r-1} - \phi\left(2p_2\cdots p_{r-1}\right)\right].$$
Consider the subgroup $S_{\frac{n}{2p_b^t}}$ of $C_n$ containing $G$. Let $\widetilde{G}$ denote the complement of $G$ in $S_{\frac{n}{2p_b^t}}$. Then $\widetilde{G}=E_{\frac{n}{2p_b^t}}\cup E_{\frac{n}{2p_b^{t+1}}}\cup \cdots \cup E_{\frac{n}{2 p_b^{n_b}}}$, a disjoint union. So
$$
|\widetilde{G}|=\underset{l=t}{\overset{n_b}{\sum}} \left\vert E_{\frac{n}{2p_b^{l}}} \right\vert =\underset{l=t}{\overset{n_b}{\sum}} \phi\left( \frac{n}{2p_b^{l}}\right)= \phi\left( \frac{n}{2p_b^{n_b}}\right)\times p_b^{n_b-t} = \frac{n}{2p_2\cdots p_r}\times \frac{1}{p_b^{t-1}}\times \phi\left(\frac{p_2p_3\cdots p_{r}}{p_b}\right).
$$
and hence
$$|G|=\left|S_{\frac{n}{2p_b^t}} \right| - |\widetilde{G}|= \frac{n}{2p_2\cdots p_r}\times \frac{1}{p_b^{t-1}}\times \left[\frac{p_2p_3\cdots p_{r}}{p_b}-\phi\left(\frac{p_2p_3\cdots p_{r}}{p_b}\right)\right].$$
Writing $\mu=\frac{n}{2p_2\cdots p_r}\times \frac{1}{p_b^{t-1}}$, we get
\begin{align*}
|G|- \left| K_{r,b}^{1,t}\right| &= \mu\times \left[\frac{p_2p_3\cdots p_{r}}{p_b} - \frac{2p_2\cdots p_{r-1}}{p_b} - \phi\left(\frac{p_2p_3\cdots p_{r}}{p_b}\right)+\frac{1}{p_b}\phi\left(2p_2\cdots p_{r-1}\right) \right]\\
 & = \mu\times \left[\frac{p_2p_3\cdots p_{r-1}}{p_b}\times(p_r-2) - \phi\left(\frac{p_2p_3\cdots p_{r-1}}{p_b}\right) \left(\phi(p_r)-\frac{\phi(p_b)}{p_b}\right)\right]\\
 & > \mu\times \left[\frac{p_2p_3\cdots p_{r-1}}{p_b}\times(p_r-2) - \phi\left(\frac{p_2p_3\cdots p_{r-1}}{p_b}\right) (p_r -1)\right].
\end{align*}
Let $k\in [r]\setminus\{1,b,r\}$. Such a $k$ exists as $r\geq 4$. Then $p_r-p_k\geq 2$ (both being odd primes) and so $p_k(p_r-2)-\phi(p_k)(p_r-1)=p_r-p_k-1 >0$. Hence it follows that
$$\left[\frac{p_2p_3\cdots p_{r-1}}{p_b}\times(p_r-2) - \phi\left(\frac{p_2p_3\cdots p_{r-1}}{p_b}\right) (p_r -1)\right] >0$$
and therefore, $|G| > \left| K_{r,b}^{1,t}\right|$. This completes the proof.
\end{proof}

Recall that the cut-set $X_{a,b}^{s,t}$ of $\mathcal{P}(C_n)$ is a disjoint union of $H_{a,b}^{s,t}$ and $K_{a,b}^{s,t}$. Let $x\in X_{a,b}^{s,t}$. Then $x\in H_{a,b}^{s,t}$ if and only if $o(x)$ is divisible by $\frac{n}{p_a^s p_b^t}$ with $o(x)\neq \frac{n}{p_a^s p_b^t}$. Further, $x\in K_{a,b}^{s,t}$ if and only if $o(x)$ divides $\frac{n}{p_a^s p_b^t}$ with $o(x)\neq \frac{n}{p_a^s p_b^t}$ if and only if $x\in S_{\frac{n}{p_ip_a^sp_b^t}}$ for some $i\in J$, where
$$J:=
\begin{cases}
[r]\setminus\{a\}, & \text{if } s=n_a,\; t<n_b\\
[r]\setminus\{b\}, & \text{if } s<n_a,\; t=n_b\\
[r], &\text{if } s<n_a,\; t<n_b
\end{cases}$$
as $(s,t)\neq (n_a, n_b)$. Thus $K_{a,b}^{s,t}$ is the union of the subgroups $S_{\frac{n}{p_ip_a^{s}p_b^{t}}}$ with $i\in J$.

Let $L_{a,b}^{s,t}$ denote the set of all elements of $K_{a,b}^{s,t}$ which do not belong to any of the subgroups $S_{\frac{n}{p_ip_a^sp_b^t}}$, where $i\in J\setminus\{a,b\}=[r]\setminus\{a,b\}$. Since $(s,t)\neq (n_a,n_b)$, we have $L_{a,b}^{s,t}\neq \emptyset$. Further, $L_{a,b}^{s,t}$ is precisely union of the mutually disjoint sets $E_{\frac{n}{p_a^kp_b^l}}$, where $s\leq k\leq n_a$, $t\leq l\leq n_b$ with $(k,l)\neq (s,t)$.

\begin{proposition}\label{L-a-b-s-t}
$\left|L_{a,b}^{s,t}\right|=\phi\left(\frac{n}{p_a^{n_a}p_b^{n_b}}\right)\left[p_a^{n_a-s}p_b^{n_b-t}-\phi\left(p_a^{n_a-s}p_b^{n_b-t}\right)\right]$.
\end{proposition}

\begin{proof}
We have $\left|L_{a,b}^{s,t}\right|=\sum \left| E_{\frac{n}{p_a^kp_b^l}}\right|$, where the sum runs over all $k,l$ satisfying $s\leq k\leq n_a$, $t\leq l\leq n_b$ and $(k,l)\neq (s,t)$. Since $\underset{k=s}{\overset{n_a}{\sum}}\; \underset{l=t}{\overset{n_b}{\sum}} \left\vert E_{\frac{n}{p_a^{k}p_b^{l}}} \right\vert =\phi\left( \frac{n}{p_a^{n_a}p_b^{n_b}}\right) p_a^{n_a-s} p_b^{n_b-t}$ by Lemma \ref{lem-2.4}, and $\left| E_{\frac{n}{p_a^{s}p_b^{t}}}\right|=\phi\left(\frac{n}{p_a^{s}p_b^{t}}\right)=\phi\left( \frac{n}{p_a^{n_a}p_b^{n_b}}\right) \phi\left(p_a^{n_a-s} p_b^{n_b-t}\right)$, the proposition follows.
\end{proof}

\begin{proposition}\label{at-most-L}
The number of elements of $X$ which do not belong to $H_{a,b}^{s,t}$ or $K_{a,b}^{s,t}\setminus L_{a,b}^{s,t}$ is at most $\left|L_{a,b}^{s,t}\right|$.
\end{proposition}

\begin{proof}
The cut-set $X_{a,b}^{s,t}$ of $\mathcal{P}(C_n)$ is a disjoint union of $H_{a,b}^{s,t}$, $L_{a,b}^{s,t}$ and $K_{a,b}^{s,t}\setminus L_{a,b}^{s,t}$. By Proposition \ref{elements-in-X}, $X$ contains the sets $H_{a,b}^{s,t}$ and $K_{a,b}^{s,t}\setminus L_{a,b}^{s,t}$. Then the proposition follows as $X$ is a minimum cut-set of $\mathcal{P}(C_n)$.
\end{proof}

\begin{proposition}\label{s-t-less}
If $s<n_a$ and $t<n_b$, then at least one of the sets $E_{\frac{n}{p_a^{s+1}}}$ and $E_{\frac{n}{p_b^{t+1}}}$ is contained in $B$.
\end{proposition}
\begin{proof}
By Proposition \ref{two-2-1}, each of $E_{\frac{n}{p_a^{s+1}}}$ and $E_{\frac{n}{p_b^{t+1}}}$ is contained in either $B$ or $X$. If possible, suppose that both $E_{\frac{n}{p_a^{s+1}}}$ and $E_{\frac{n}{p_b^{t+1}}}$  are contained in $X$. Note that $E_{\frac{n}{p_a^{s+1}}}$ and $E_{\frac{n}{p_b^{t+1}}}$ are disjoint and each of them is disjoint from the cut-set $X_{a,b}^{s,t}$ of $\mathcal{P}(C_n)$. By proposition \ref{at-most-L}, we have
$\left|E_{\frac{n}{p_a^{s+1}}}\right|+ \left|E_{\frac{n}{p_b^{t+1}}}\right|\leq \left|L_{a,b}^{s,t}\right|$,
that is,
\begin{equation}\label{ineq-E-L}
\xi:=\left|E_{\frac{n}{p_a^{s+1}}}\right| + \left|E_{\frac{n}{p_b^{t+1}}}\right| -\left|L_{a,b}^{s,t}\right|\leq 0.
\end{equation}
By Proposition \ref{L-a-b-s-t}, we have
$\left|L_{a,b}^{s,t}\right|=\phi\left(\frac{n}{p_a^{n_a}p_b^{n_b}}\right)\left[p_a^{n_a-s}p_b^{n_b-t}-\phi\left(p_a^{n_a-s}p_b^{n_b-t}\right)\right].$
Further,
$$\left|E_{\frac{n}{p_a^{s+1}}}\right|=\phi\left( \frac{n}{p_a^{s+1}}\right)=\phi\left(\frac{n}{p_a^{n_a}p_b^{n_b}}\right)\phi\left(p_a^{n_a-s-1}p_b^{n_b}\right)$$
and 
$$\left|E_{\frac{n}{p_b^{t+1}}}\right|=\phi\left( \frac{n}{p_b^{t+1}}\right)=\phi\left(\frac{n}{p_a^{n_a}p_b^{n_b}}\right)\phi\left(p_a^{n_a}p_b^{n_b-t-1}\right).$$
So
$$\xi =\phi\left(\frac{n}{p_a^{n_a}p_b^{n_b}}\right) \left[\phi\left(p_a^{n_a-s-1}p_b^{n_b}\right) + \phi\left(p_a^{n_a}p_b^{n_b-t-1}\right) +\phi\left(p_a^{n_a-s}p_b^{n_b-t}\right)-p_a^{n_a-s}p_b^{n_b-t}\right].$$
We show that $\xi >0$, thus getting a contradiction to (\ref{ineq-E-L}). We consider four different cases depending on $s$ and $t$.\medskip

\noindent (i) If $s+1 < n_a$ and $t+1 < n_b$, then
\begin{align*}
\xi &=\phi\left(\frac{n}{p_a^{n_a}p_b^{n_b}}\right)\times\frac{p_a^{n_a}p_b^{n_b}}{p_a^{s+1}p_b^{t+1}}\times\left[\phi(p_ap_b)\left(\frac{p_b^t}{p_a} + \frac{p_a^s}{p_b} +1\right)-p_ap_b\right].
\end{align*}
Since $\frac{p_b^t}{p_a} + \frac{p_a^s}{p_b}=\frac{p_b^{t+1}+p_a^{s+1}}{p_ap_b}\geq \frac{p_b^2+p_a^2}{p_ap_b}> 2$, it follows that $\xi >0$.\medskip

\noindent (ii) If $s+1 = n_a$ and $t+1 < n_b$, then
\begin{align*}
\xi & =\phi\left(\frac{n}{p_a^{n_a}p_b^{n_b}}\right) \left[\phi\left(p_b^{n_b}\right) + \phi\left(p_a^{n_a}p_b^{n_b-t-1}\right) +\phi\left(p_ap_b^{n_b-t}\right)-p_ap_b^{n_b-t}\right]\\
&=\phi\left(\frac{n}{p_a^{n_a}p_b^{n_b}}\right)\times\frac{p_b^{n_b}}{p_b^{t+1}}\times\left[\phi(p_ap_b)\left(\frac{p_b^t}{\phi(p_a)} + \frac{p_a^{n_a-1}}{p_b} +1\right)-p_ap_b\right].
\end{align*}
Since $\frac{p_b^t}{\phi(p_a)} + \frac{p_a^{n_a-1}}{p_b} > \frac{p_b^t}{p_a} + \frac{p_a^{n_a-1}}{p_b}=\frac{p_b^{t+1}+p_a^{n_a}}{p_ap_b}\geq \frac{p_b^2+p_a^2}{p_ap_b} > 2$, it follows that $\xi >0$.\medskip

\noindent (iii) If $s+1 < n_a$ and $t+1 = n_b$, then
\begin{align*}
\xi & =\phi\left(\frac{n}{p_a^{n_a}p_b^{n_b}}\right) \left[\phi\left(p_a^{n_a-s-1}p_b^{n_b}\right) + \phi\left(p_a^{n_a}\right) + \phi\left(p_a^{n_a-s}p_b\right) -p_a^{n_a-s}p_b\right]\\
&=\phi\left(\frac{n}{p_a^{n_a}p_b^{n_b}}\right)\times\frac{p_a^{n_a}}{p_a^{s+1}}\times\left[\phi(p_ap_b)\left(\frac{p_b^{n_b-1}}{p_a} + \frac{p_a^{s}}{\phi(p_b)} +1\right)-p_ap_b\right].
\end{align*}
Since $\frac{p_b^{n_b-1}}{p_a} + \frac{p_a^{s}}{\phi(p_b)} > \frac{p_b^{n_b-1}}{p_a} + \frac{p_a^{s}}{p_b}=\frac{p_b^{n_b}+p_a^{s+1}}{p_ap_b}\geq \frac{p_b^2+p_a^2}{p_ap_b} > 2$, it follows that $\xi >0$.\medskip

\noindent (iv) If $s+1 = n_a$ and $t+1 = n_b$, then
\begin{align*}
\xi & =\phi\left(\frac{n}{p_a^{n_a}p_b^{n_b}}\right) \left[\phi\left(p_b^{n_b}\right) + \phi\left(p_a^{n_a}\right) + \phi\left(p_ap_b\right) -p_ap_b\right]\\
&=\phi\left(\frac{n}{p_a^{n_a}p_b^{n_b}}\right)\left[\phi(p_ap_b)\left(\frac{p_b^{n_b-1}}{\phi(p_a)} + \frac{p_a^{n_a-1}}{\phi(p_b)} +1\right)-p_ap_b\right].
\end{align*}
Since $\frac{p_b^{n_b-1}}{\phi(p_a)} + \frac{p_a^{n_a-1}}{\phi(p_b)} > \frac{p_b^{n_b-1}}{p_a} + \frac{p_a^{n_a-1}}{p_b}=\frac{p_b^{n_b}+p_a^{n_a}}{p_ap_b}\geq \frac{p_b^2+p_a^2}{p_ap_b} > 2$, it follows that $\xi >0$.
\end{proof}

Let $M_{a,b}^{s,t}$ denote the set of elements of $K_{a,b}^{s,t}$ which do not belong to any of the subgroups $S_{\frac{n}{p_ip_a^sp_b^t}}$, where $i\in J\setminus\{a\}$. We have $M_{a,b}^{s,t}=\emptyset$ if and only if $s=n_a$. If $s < n_a$, then $M_{a,b}^{s,t}$ is a subset of $S_{\frac{n}{p_a^{s+1}p_b^t}}$ and is precisely union of the mutually disjoint sets $E_{\frac{n}{p_a^kp_b^t}}$, where $s+1\leq k\leq n_a$.

Similarly, let $N_{a,b}^{s,t}$ denote the set of elements of $K_{a,b}^{s,t}$ which do not belong to any of the subgroups $S_{\frac{n}{p_ip_a^sp_b^t}}$, where $i\in J\setminus\{b\}$. We have $N_{a,b}^{s,t}=\emptyset$ if and only if $t=n_b$. If $t < n_b$, then $N_{a,b}^{s,t}$ is a subset of $S_{\frac{n}{p_a^{s}p_b^{t+1}}}$ and is precisely union of the mutually disjoint sets $E_{\frac{n}{p_a^sp_b^l}}$, where $t+1\leq l\leq n_b$.

The proof of the following proposition is straightforward.

\begin{proposition}\label{M-N-a-b-s-t}
The following hold:
\begin{enumerate}
\item[(i)] If $s<n_a$, then $\left|M_{a,b}^{s,t}\right|=\phi\left(\frac{n}{p_a^{n_a}p_b^{t}}\right) p_a^{n_a-s-1}$.

\item[(ii)] If $t<n_b$, then $\left|N_{a,b}^{s,t}\right|=\phi\left(\frac{n}{p_a^{s}p_b^{n_b}}\right) p_b^{n_b-t-1}$.
\end{enumerate}
\end{proposition}

\begin{proposition}\label{s-less}
If $s<n_a$, then the set $E_{\frac{n}{p_a^{s+1}}}$ is contained in $B$.
\end{proposition}
\begin{proof}
We have $p_a<p_b$ as $a<b$ by our assumption. By Proposition \ref{two-2-1}, the set $E_{\frac{n}{p_a^{s+1}}}$ is contained in either $B$ or $X$. If possible, suppose that $E_{\frac{n}{p_a^{s+1}}}$ is contained in $X$.

If $t=n_b$, then $J=[r]\setminus\{b\}$ and so $J\setminus\{a\}=[r]\setminus\{a,b\}$. In this case, the subgroups $S_{\frac{n}{p_ip_a^sp_b^t}}$, $i\in J\setminus\{a\}$, are contained in $X$ by Proposition \ref{elements-in-X}. If $t<n_b$, then the set $E_{\frac{n}{p_b^{t+1}}}$ is contained in $B$ by Proposition \ref{s-t-less}. Since $E_{\frac{n}{p_a^sp_b^t}}$ is contained in $A$, the subgroup $S_{\frac{n}{p_a^sp_b^{t+1}}}$ must be contained in $X$. In this case also, the subgroups $S_{\frac{n}{p_ip_a^sp_b^t}}$, $i\in J\setminus\{a\}$, are contained in $X$. Thus $X$ contains $K_{a,b}^{s,t}\setminus M_{a,b}^{s,t}$ in both cases.

The cut-set $X_{a,b}^{s,t}$ of $\mathcal{P}(C_n)$ is a disjoint union of $H_{a,b}^{s,t}$, $M_{a,b}^{s,t}$ and $K_{a,b}^{s,t}\setminus M_{a,b}^{s,t}$. Since $X$ is a minimum cut-set of $\mathcal{P}(C_n)$ containing  $H_{a,b}^{s,t}$ and $K_{a,b}^{s,t}\setminus M_{a,b}^{s,t}$, it follows that the number of elements of $X$ which do not belong to $H_{a,b}^{s,t}$ or $K_{a,b}^{s,t}\setminus M_{a,b}^{s,t}$ is at most $\left|M_{a,b}^{s,t}\right|$. In particular, as $E_{\frac{n}{p_a^{s+1}}}$ is disjoint from the cut-set $X_{a,b}^{s,t}$, we must have $\left|E_{\frac{n}{p_a^{s+1}}}\right|\leq \left|M_{a,b}^{s,t}\right|$, that is,
\begin{equation}\label{ineq-E-M}
\xi_1:=\left|E_{\frac{n}{p_a^{s+1}}}\right| - \left|M_{a,b}^{s,t}\right|\leq 0.
\end{equation}
We have $\left|M_{a,b}^{s,t}\right|=\phi\left(\frac{n}{p_a^{n_a}p_b^{t}}\right) p_a^{n_a-s-1}$ by Proposition \ref{M-N-a-b-s-t}(i) and $\left|E_{\frac{n}{p_a^{s+1}}}\right|=\phi\left( \frac{n}{p_a^{s+1}}\right)=\phi\left(\frac{n}{p_a^{n_a}p_b^{n_b}}\right)\phi\left(p_a^{n_a-s-1}p_b^{n_b}\right)$. So
$$\xi_1 =\phi\left(\frac{n}{p_a^{n_a}p_b^{n_b}}\right) \left[\phi\left(p_a^{n_a-s-1}p_b^{n_b}\right) - p_a^{n_a-s-1}\phi\left(p_b^{n_b-t}\right)\right].$$

We consider four different cases depending on $s+1=n_a$ or $s+1<n_a$, and $t=n_b$ or $t<n_b$.
If $s+1 < n_a$ and $t< n_b$, then $\xi_1 =\phi\left(\frac{n}{p_a^{n_a}p_b^{n_b}}\right) \times p_a^{n_a-s-2}p_b^{n_b-t-1}\times \phi(p_b)\left[p_b^t\phi(p_a)-p_a\right]>0$.
If $s+1 = n_a$ and $t< n_b$, then $\xi_1 =\phi\left(\frac{n}{p_a^{n_a}p_b^{n_b}}\right) \left[\phi\left(p_b^{n_b}\right) - \phi\left(p_b^{n_b-t}\right)\right]>0$.
If $s+1 = n_a$ and $t =n_b$, then $p_b\geq 3$ gives $\xi_1 =\phi\left(\frac{n}{p_a^{n_a}p_b^{n_b}}\right) \left[\phi\left(p_b^{n_b}\right) - 1 \right]>0$. In these three cases, we get a contradiction to (\ref{ineq-E-M}).

Now consider $s+1 < n_a$ and $t= n_b$. Then $\xi_1 =\phi\left(\frac{n}{p_a^{n_a}p_b^{n_b}}\right) \left[\phi\left(p_a^{n_a-s-1}p_b^{n_b}\right) - p_a^{n_a-s-1}\right]\geq 0$,
with equality if and only if $(p_a,p_b,t)=(p_a,p_b,n_b)=(2,3,1)$. Thus we get a contradiction to (\ref{ineq-E-M}) if $(p_a,p_b,n_b)\neq (2,3,1)$. Assume that $(p_a,p_b,n_b) =(2,3,1)$. Then $\xi_1=0$, giving $\left|E_{\frac{n}{p_a^{s+1}}}\right| = \left|M_{a,b}^{s,t}\right|$. So
\begin{align*}
|X| & \geq \left|H_{a,b}^{s,t}\right| + \left|K_{a,b}^{s,t}\setminus M_{a,b}^{s,t}\right|+\left|E_{\frac{n}{p_a^{s+1}}}\right|=\left|H_{a,b}^{s,t}\right| + \left|K_{a,b}^{s,t}\setminus M_{a,b}^{s,t}\right|+\left|M_{a,b}^{s,t}\right| =\left|X_{a,b}^{s,t}\right|.
\end{align*}
By the minimality of $|X|$, we must have $X =H_{a,b}^{s,t} \bigcup \left(K_{a,b}^{s,t}\setminus M_{a,b}^{s,t}\right) \bigcup E_{\frac{n}{p_a^{s+1}}}$.

On the other hand, as $s+2\leq n_a$, the set $E_{\frac{n}{p_a^{s+2}}}$ is contained in either $X$ or $B$ by Proposition \ref{two-2-1}. If the latter holds, then the subgroup $S_{\frac{n}{p_a^{s+2}p_b}}$ would be contained in $X$. In both cases, we find that $X$ contains additional elements of order $\frac{n}{p_a^{s+2}}$ or $\frac{n}{p_a^{s+2}p_b}$ not belonging to any of the sets $H_{a,b}^{s,t}$, $K_{a,b}^{s,t}\setminus M_{a,b}^{s,t}$ and $E_{\frac{n}{p_a^{s+1}}}$, a final contradiction.
\end{proof}

\begin{proposition}\label{t-less}
If $t<n_b$, then the set $E_{\frac{n}{p_b^{t+1}}}$ is contained in $B$.
\end{proposition}
\begin{proof}
We shall apply a similar argument as used in the proof of Proposition \ref{s-less}. By Proposition \ref{two-2-1}, the set $E_{\frac{n}{p_b^{t+1}}}$ is contained in either $B$ or $X$. If possible, suppose that $E_{\frac{n}{p_b^{t+1}}}$ is contained in $X$.

If $s=n_a$, then $J=[r]\setminus\{a\}$ and so $J\setminus\{b\}=[r]\setminus\{a,b\}$. In this case, the subgroups $S_{\frac{n}{p_ip_a^sp_b^t}}$, $i\in J\setminus\{b\}$, are contained in $X$ by Proposition \ref{elements-in-X}.
If $s<n_a$, then the set $E_{\frac{n}{p_a^{s+1}}}$ is contained in $B$ by Proposition \ref{s-less} and so the subgroup $S_{\frac{n}{p_a^{s+1}p_b^{t}}}$ must be contained in $X$. In this case also, the subgroups $S_{\frac{n}{p_ip_a^sp_b^t}}$, $i\in J\setminus\{b\}$, are contained in $X$. Thus $X$ contains $K_{a,b}^{s,t}\setminus N_{a,b}^{s,t}$ in both cases.

The cut-set $X_{a,b}^{s,t}$ of $\mathcal{P}(C_n)$ is a disjoint union of $H_{a,b}^{s,t}$, $N_{a,b}^{s,t}$ and $K_{a,b}^{s,t}\setminus N_{a,b}^{s,t}$. Since $X$ is a minimum cut-set of $\mathcal{P}(C_n)$ containing $H_{a,b}^{s,t}$ and $K_{a,b}^{s,t}\setminus N_{a,b}^{s,t}$, it follows that the number of elements of $X$ which do not belong to $H_{a,b}^{s,t}$ or $K_{a,b}^{s,t}\setminus N_{a,b}^{s,t}$ is at most $\left|N_{a,b}^{s,t}\right|$. In particular, as $E_{\frac{n}{p_b^{t+1}}}$ is disjoint from the cut-set $X_{a,b}^{s,t}$, we must have $\left|E_{\frac{n}{p_b^{t+1}}}\right|\leq \left|N_{a,b}^{s,t}\right|$, that is,
\begin{equation}\label{ineq-E-N}
\xi_2:=\left|E_{\frac{n}{p_b^{t+1}}}\right| - \left|N_{a,b}^{s,t}\right| \leq 0.
\end{equation}
We have $\left|N_{a,b}^{s,t}\right|=\phi\left(\frac{n}{p_a^{s}p_b^{n_b}}\right) p_b^{n_b-t-1}$ by Proposition \ref{M-N-a-b-s-t}(ii) and $\left|E_{\frac{n}{p_b^{t+1}}}\right|= \phi\left(\frac{n}{p_b^{t+1}}\right) =\phi\left(\frac{n}{p_a^{n_a}p_b^{n_b}}\right)\phi\left(p_a^{n_a}p_b^{n_b-t-1}\right)$. So
$$\xi_2 =\phi\left(\frac{n}{p_a^{n_a}p_b^{n_b}}\right) \left[\phi\left(p_a^{n_a}p_b^{n_b-t-1}\right) - \phi\left(p_a^{n_a-s}\right)p_b^{n_b-t-1}\right]. $$
We show that $\xi_2 > 0$, thus getting a contradiction to (\ref{ineq-E-N}).

Since $t<n_b$, we have $(p_a,n_a)\neq (2,1)$ by Proposition \ref{not-(2,1)}. We consider four different cases depending on $s=n_a$ or $s<n_a$, and $t+1=n_b$ or $t+1<n_b$.

If $s < n_a$ and $t+1< n_b$, then $\xi_2 =\phi\left(\frac{n}{p_a^{n_a}p_b^{n_b}}\right) \times p_a^{n_a-s-1}p_b^{n_b-t-2}\times \phi(p_a)\left[p_a^s\phi(p_b)-p_b\right]>0$.
If $s < n_a$ and $t+1 = n_b$, then $\xi_2 =\phi\left(\frac{n}{p_a^{n_a}p_b^{n_b}}\right) \left[\phi\left(p_a^{n_a}\right) - \phi\left(p_a^{n_a-s}\right)\right]>0$.
If $s = n_a$ and $t+1=n_b$, then $\xi_2 =\phi\left(\frac{n}{p_a^{n_a}p_b^{n_b}}\right) \left[\phi\left(p_a^{n_a}\right) - 1 \right] > 0$ as $(p_a,n_a)\neq (2,1)$.
If $s= n_a$ and $t+1 < n_b$, then $\xi_2 =\phi\left(\frac{n}{p_a^{n_a}p_b^{n_b}}\right) \left[\phi\left(p_a^{n_a}p_b^{n_b-t-1}\right) - p_b^{n_b-t-1}\right]> 0$ as $(p_a,n_a)\neq (2,1)$.
\end{proof}

\begin{proposition}
$X=X_{a,b}^{s,t}$.
\end{proposition}
\begin{proof}
The set $E_{\frac{n}{p_a^sp_b^t}}$ is contained in $A$ and the sets $E_{\frac{n}{p_i}}$, $i\in [r]\setminus\{a,b\}$, are contained in $B$. By Corollary \ref{s=n-a t=n-b}, we may assume that $(s,t)\neq (n_a,n_b)$. If $s<n_a$, then $E_{\frac{n}{p_a^{s+1}}}$ is contained in $B$ by Proposition \ref{s-less}. If $t<n_b$, then $E_{\frac{n}{p_b^{t+1}}}$ is contained in $B$ by Proposition \ref{t-less}. It then follows that the subgroups $S_{\frac{n}{p_ip_a^sp_b^t}}$ for $i\in J$ and hence the set $K_{a,b}^{s,t}$ must be contained in $X$. By Proposition \ref{elements-in-X}, $H_{a,b}^{s,t}$ is contained in $X$. Thus the cut-set $X_{a,b}^{s,t}$ of $\mathcal{P}(C_n)$ is contained in $X$. Since $X$ is a minimum cut-set of $\mathcal{P}(C_n)$, we must have $X=X_{a,b}^{s,t}$.
\end{proof}

\textbf{Acknowledgement:}

The first author was supported by the Council of Scientific and Industrial Research grant no. 09/1248(0004)/2019-EMR-I, Ministry of Human Resource Development, Government of India.

\vskip .5cm

\noindent{\bf Addresses}: Sanjay Mukherjee, Kamal Lochan Patra, Binod Kumar Sahoo

\begin{enumerate}
\item[1)] School of Mathematical Sciences\\
National Institute of Science Education and Research (NISER), Bhubaneswar\\
P.O.- Jatni, District- Khurda, Odisha--752050, India

\item[2)] Homi Bhabha National Institute (HBNI)\\
Training School Complex, Anushakti Nagar, Mumbai--400094, India
\end{enumerate}

\noindent {\bf E-mails}: sanjay.mukherjee@niser.ac.in, klpatra@niser.ac.in, bksahoo@niser.ac.in
\end{document}